\documentclass[12pt]{article}
\usepackage{graphicx}
\usepackage{epsfig}
\usepackage{fancyhdr}

\usepackage{amsmath,amsfonts,amssymb,amsthm}
\theoremstyle{plain}
\newtheorem{thm}{Theorem}
\newtheorem*{thm*}{Theorem}

\newtheorem*{lem*}{Lemma}
\newtheorem*{cor}{Corollary}

\theoremstyle{definition}

\newtheorem{Def}{Definition}

\newcommand{\reftit}{\textit}    % reference title
\newcommand{\refis}{\textbf}     % reference issue no.

\begin{document}

\title{A note on adiabatic theorem for Markov chains and adiabatic quantum computation}
\author{Yevgeniy Kovchegov\footnote{
 Department of Mathematics,  Oregon State University, Corvallis, OR  97331-4605, USA
 \texttt{kovchegy@math.oregonstate.edu}}}
\date{ }
\maketitle

\abstract{We derive an adiabatic theorem for Markov chains using well known facts about mixing and relaxation times. We discuss the results in the context of the recent developments in adiabatic quantum computation.}

\section{Introduction}
 The adiabatic theorem was first stated by  Max Born and Vladimir Fock  in 1928. It asserts that ``a physical system remains in its instantaneous eigenstate if a given perturbation is acting on it slowly enough and if there is a gap between the eigenvalue and the rest of the Hamiltonian's spectrum" (see Wikipedia page on adiabatic theorem).

 Since under certain conditions, there is a way of expressing Hamiltonians via reversible Markov chains, we state and prove the corresponding theorem for Markov chains. The task turned out to be less complicated than it was initially expected, and we used only a small fraction of  the machinery of mixing times and relaxation times of Markov chains that was developed so successfully in the last thirty years (see \cite{a}, \cite{af}, \cite{bk}, \cite{d}, \cite{lpw}, \cite{sj} and references therein). Stated in terms of Markov chains, the adiabatic theorem is intuitively simple and accessible (see \cite{ar} for comparison). Here we will quote Nielsen and Chuang \cite{nc}: ``One of the goals of quantum computation and quantum information is to develop tools which sharpen our intuition about quantum mechanics, and make its predictions more transparent to human minds". It is important to point out that despite clear similarities and relation between the two, the quantum adiabatic theorem and the adiabatic theorem of this paper are different.
 
\vskip 0.3 in
 First we state a version of the adiabatic theorem.
 Given two Hamiltonians, ${\cal H}_{initial}$ and ${\cal H}_{final}$, acting on a quantum system. 
 Let ${\cal H}(s)=(1-s){\cal H}_{initial}+s{\cal H}_{final}$.
 Suppose the system evolves according to ${\cal H}(t/T)$ from time $t=0$ to time $T$ (the so called {\it adiabatic evolution}).
The adiabatic theorem of quantum mechanics states that for $T$ large enough, the final state of the system will be close to
the ground state of ${\cal H}_{final}$. They are $\epsilon$ close in $l_2$ norm whenever $T \geq {C \over \epsilon \beta^3}$,
where $\beta$ is the least spectral gap of ${\cal H}(s)$ over all $s \in [0,1]$, and $C$ depends linearly on a square of the distance between ${\cal H}_{initial}$ and ${\cal H}_{final}$.
 See \cite{jrs} and \cite{m}. There are many versions of the adiabatic theorem.

\vskip 0.3 in
\noindent
Now, let us list the main concepts involved in the theory of mixing and relaxation times that we will need in the following sections. We refer the readers to \cite{lpw} and \cite{af} for details.  
\begin{Def}
If $\mu$ and $\nu$ are two probability distributions over $\Omega$, then the {\it total variation distance} is 
$$\| \mu - \nu \|_{TV} = {1 \over 2} \sum_{x \in \Omega} |\mu(x)-\nu(x)|=\sup_{A \subset \Omega} |\mu(A)-\nu(A)|$$
Observe that the total variation distance measures the coincidence between the distributions on a scale from zero to one.
\end{Def}

\begin{Def}
 Suppose $P$ is an irreducible and aperiodic finite Markov chain with stationary distribution $\pi$, i.e. $\pi P=\pi$. Given an $\epsilon >0$, the mixing time $t_{mix}(\epsilon)$ is defined as
 $$t_{mix}(\epsilon)=\inf\left\{t~:~\|\nu P^t - \pi\|_{TV} \leq \epsilon,~~\text{ for all probability distributions } \nu \right\}$$
\end{Def}
 
 \vskip 0.2 in
 \noindent
 Now, suppose $P=\Big( p(x,y) \Big)_{x,y}$ is {\it reversible}, i.e. $P$ satisfies the detailed balance condition
 $$\pi(x)p(x,y)=\pi(y) p(y,x)~~\text{ for all }x,y~\text{ in the sample space }\Omega$$
 It is easy to check that if $P$ is reversible, then $P$ is self-adjoint with respect to an inner product induced by $\pi$, and as such will have all real eigenvalues. 
 The difference $\beta=1-|\lambda_{(2)}|$ between the largest eigenvalue (i.e. $\lambda_1=1$) and the second largest (in absolute value) eigenvalue $\lambda_{(2)}$ is called the {\it spectral gap} of $P$. The relaxation time is defined as
 $$\tau_{rlx}={1 \over \beta}$$
 One can show the following relationship between mixing and relaxation times.
 \begin{thm} \label{relax}
 Suppose $P$ is a reversible, irreducible and aperiodic Markov chain with state space $\Omega$ and stationary distribution $\pi$. 
 Then
 $$(\tau_{rlx}-1)\log(2\epsilon)^{-1} \leq t_{mix}(\epsilon) \leq \tau_{rlx} \log (\epsilon \min_{x \in \Omega}\pi(x) )^{-1}$$
 \end{thm}
 \vskip 0.3 in
 To conclude the introduction, we would like to point out that the mixing and relaxation time techniques were widely used in the the theory of computations and algorithms (see \cite{sj}).  This current paper was influenced by the work of D.Aharonov et al \cite{siam}. The model of adiabatic quantum computation is based on the quantum adiabatic theorem. The connection between the adiabatic quantum computation and relaxation times of reversible Markov chains was used in \cite{siam} to prove an equivalence between adiabatic quantum and quantum computations. See \cite{ksv} and \cite{nc} for more on quantum computation.

\section{Hamiltonians, Markov chains, and adiabatic quantum computation}

 Let the Hamiltonian (or {\it energy}) ${\cal H}$ be a Hermitian real matrix. If it is also  true that $I-{\cal H}$ is nonnegative, irreducible and aperiodic (i.e. $(I-{\cal H})^k>0$ for $k$ large enough), by Perron's theorem, its largest (in absolute value) eigenvalue is positive with a unique eigenvector with all positive coordinates. 
 
 Let $I-{\cal H}=\Big(a_{i,j}\Big)$ be an irreducible and aperiodic nonnegative $(N+1) \times(N+1)$ dimensional Hamiltonian matrix, with eigenvalues $$1-\lambda_0 >1-\lambda_1 \geq 1-\lambda_2 \geq \dots \geq 1-\lambda_N,$$ where
 $\lambda_j$ are eigenvalues of the Hamiltonian ${\cal H}$, the lowest of them is the {\it ground energy} $\lambda_0$, and by Perron's theorem, $1-\lambda_0>|1-\lambda_j|~~(j=1,\dots,N)$. Then $$P=\left(P_{i,j}={\alpha_j \over (1-\lambda_0) \alpha_i}a_{i,j}\right),$$
 where the  {\it ground state} ${\bf \alpha}=(\alpha_0,\dots,\alpha_n)$ with all $\alpha_j >0$ is the unique eigenvector of ${\cal H}$ corresponding to $\lambda_0$,
 will be an irreducible aperiodic reversible Markov chain over $(N+1)$ states,  with eigenvalues
 $$r_0=1>r_1 \geq r_2 \geq \dots \geq r_N,$$
 where $r_j={1-\lambda_j \over 1-\lambda_0}$, and stationary distribution 
 $$\pi={1 \over \|{\bf \alpha}\|_{\ell_2}^2}(\alpha_0^2,\dots,\alpha_N^2)$$
 Thus if $I-{\cal H}$ nonnegative, irreducible and aperiodic, we can convert the Hamiltonian into a reversible transition probability matrix.
 Observe that the spectral gap $\beta=1-\max\{|r_1|,|r_N|\}$ of $P$ will be a multiple of the spectral gap 
 $\beta_H=\lambda_1-\lambda_0$ of ${\cal H}$, $$\beta={\beta_H \over 1-\lambda_0}$$
 whenever $r_1 \geq |r_N|$. If we consider a {\it lazy Markov chain} $\widetilde{P}={1 \over 2}(I+P)$, then the eigenvectors of $\widetilde{P}$  will coincide with the eigenvectors of $P$, and its eigenvalues will be 
 $$\tilde{r}_0=1>\tilde{r}_1 \geq \tilde{r}_2 \geq \dots \geq \tilde{r}_N \geq 0,$$
 where $\tilde{r}_j={1+r_j \over 2}$. The lazy Markov chain $\widetilde{P}$ evolves twice slower than $P$, and its spectral gap
 $$\widetilde{\beta}=1-\max\{|\tilde{r}_1|,|\tilde{r}_N|\}=1-\tilde{r}_1={\beta_H \over 2(1-\lambda_0)}$$
 Observe that $$P={1 \over 1-\lambda_0}(I- A^{-1}\mathcal{H}A),$$
 where $$A=\left(\begin{array}{cccc}\alpha_0 & 0 & \dots & 0 \\0 & \alpha_1 & \ddots & 0 \\\vdots & \ddots & \ddots & 0 \\0 & \dots & \dots & \alpha_N\end{array}\right)$$

 \noindent
  In quantum computation, if we operate on $n$ {\it qubits} (quantum bits), the Hamiltonian will be a $2^n \times 2^n$ matrix. In this case, if the conditions are satisfied, and we can convert the Hamiltonian into a Markov chain, the stationary distribution $\pi$ will represent the quantum probabilities for the multiple qubits, $|0\dots00>,~|0\dots01>,\dots,|1\dots11>$
  We can also consider a restriction of the Hamiltonian to an invariant subspace, that will have the needed properties.
 
 \vskip 0.2 in
 In the model of adiabatic quantum computation, an algorithm is described by the adiabatic evolution
 $${\cal H}(s)=(1-s){\cal H}_{initial}+s{\cal H}_{final},$$
 where the input is the {\it ground state} (the eigenvector corresponding to the lowest eigenvalue of)  ${\cal H}_{initial}$, and
 the desired output is the ground state of ${\cal H}_{final}$. Thus, according to the quantum adiabatic theorem, if the duration time $T$  is large enough, then the output will be $\epsilon$ close to the ground state of ${\cal H}_{final}$. The duration time $(T\max_{s}\|{\cal H}(s)\|)$ required will be the running time for the algorithm (according to some definitions). Therefore finding the optimal $T_{\epsilon}$ would optimize the algorithm's running time. In \cite{siam}, D.Aharonov et al use Markov chain techniques to estimate the spectral gap for the Markov chain $P_{final}$ generated from ${\cal H}_{final}$, and therefore, obtaining the estimate for the spectral gap of  ${\cal H}(s)$ when $s$ is close enough to one, and subsequently using the quantum adiabatic theorem to bound the running time. 

 \vskip 0.2 in
 \noindent
 Now, suppose $\mathcal{H}$ is a Hermitian real matrix that satisfies the required conditions, i.e.  $(I-{\cal H})^k>0$ for $k$ large enough.
 Then $(1-\lambda_0)P=I- A^{-1}\mathcal{H}A$ and
 $$\mathcal{H}=I-(1-\lambda_0)D^{-{1 \over 2}}\widehat{P}D^{1 \over 2},$$
 where
 $$D=\left(\begin{array}{cccc} \pi_0 & 0 & \dots & 0 \\0 & \pi_1 & \ddots & 0 \\\vdots & \ddots & \ddots & 0 \\0 & \dots & \dots & \pi_N\end{array}\right)={1 \over \|{\bf \alpha}\|_{\ell_2}^2}A^2~~~\text{ and }~~\widehat{P}=DPD^{-1}$$
Here, since $\pi=(\pi_0,\dots,\pi_N)$ is the stationary distribution of reversible $P$, the matrix $\widehat{P}$ will be a self-adjoint Markov chain with the same spectrum as $P$.  
 \vskip 0.2 in
 \noindent
 In the above discussion we summarized a connection between the Hamiltonians and Markov chains in the context of quantum information.

\section{Adiabatic theorem for Markov chains}
 
  Given two  transition probability operators, $P_{initial}$ and $P_{final}$, with a finite state space $\Omega$. Suppose $P_{final}$ is irreducible and aperiodic and $\pi_f$ is the unique stationary distribution for $P_{final}$. Let 
  $$P_s=(1-s)P_{initial}+sP_{final}$$
  \begin{Def}
  Given $\epsilon>0$, a time $T_{\epsilon}$ is called the {\it adiabatic time} if it is the least $T$ such that
  $$\max_{\nu} \| \nu P_{1 \over T} P_{2 \over T} \cdots P_{T-1 \over T} P_1 - \pi_f \|_{TV} \leq \epsilon,$$
  where the maximum is taken over all probability distributions $\nu$ over $\Omega$.  
  \end{Def} 
  \begin{thm} \label{adiabatic1} Let $t_{mix}$ denote the mixing time for $P_{final}$. Then the adiabatic time
  $$T_{\epsilon}=O\left({t_{mix}(\epsilon/2)^2 \over \epsilon} \right)$$
  \end{thm}
  \begin{proof}
  Observe that
  $$\nu P_{1 \over T} P_{2 \over T} \cdots P_{T-1 \over T} P_1={T! \over N! T^{T-N}}~~\nu_N P_{final}^{T-N}+ \mathcal{E},$$
  where $\nu_N=\nu P_{1 \over T} P_{2 \over T} \cdots P_{N \over T}$ and $\mathcal{E}$ is the rest of the terms.
  Hence, by triangle inequality,
  \begin{eqnarray*}
  \max_{\nu} \| \nu P_{1 \over T} P_{2 \over T} \cdots P_{T-1 \over T} P_1 - \pi_f \|_{TV} 
  & \leq & 
  %{T! \over T^T} \sum_{j=0}^{N} \left(1-{j \over T}\right)\cdot {T^j\over j! }  \max_{\nu} \| \nu P_{final}^{T-j} - \pi_f \|_{TV}+S_N\\
  %& \leq & \max_{\nu} \| \nu P_{final}^{T-N}  - \pi_f \|_{TV}\cdot {T! \over T^T} \sum_{j=0}^{N} \left(1-{j \over T}\right)\cdot {T^j\over j! } +S_N\\ & = & 
  \max_{\nu} \| \nu P_{final}^{T-N}  - \pi_f \|_{TV}\cdot {T! \over N! T^{T-N}} +S_N,
  \end{eqnarray*}
where $0 \leq S_N \leq 1-{T! \over N! T^{T-N}}$.

 \vskip 0.2 in
 \noindent
 Let $T=K t_{mix}(\epsilon/2)$ and $N=(K-1) t_{mix}(\epsilon/2)$, so that 
 $ \max_{\nu} \| \nu P_{final}^{T-N}  - \pi_f \|_{TV} \leq \epsilon /2$. Observe that 
 $$e^{\int_{N}^T \log{x} dx+(T-N)\log{T}} \leq  {T! \over N! T^{T-N}} =e^{\sum_{j=N+1}^T \log{j}+(T-N)\log{T}} \leq 1$$
 and therefore, expressing $T$ and $N$ via $t_{mix}(\epsilon/2)$, and simplifying, we obtain
 $$\left( {\left(1 +{1 \over K-1} \right)^{K-1} \over e} \right)^{t_{mix}(\epsilon/2)}=e^{N\log{T \over N}-(T-N)} \leq  {T! \over N! T^{T-N}} \leq 1 $$   
 So
 $$0 \leq S_N \leq 1-{T! \over N! T^{T-N}} \leq 1-\left( {\left(1 +{1 \over K-1} \right)^{K-1} \over e} \right)^{t_{mix}(\epsilon/2)}$$
  and we need to find the least $K$ such that
  $$1-\left( {\left(1 +{1 \over K-1} \right)^{K-1} \over e} \right)^{t_{mix}(\epsilon/2)} \leq \epsilon /2$$
 Now, since $\log(1+x)=x-{x^2 \over 2}+O(x^3)$,  the least such $K$ is approximated as follows
 $$K \approx {t_{mix}(\epsilon/2) \over -2 \log\left(1-\epsilon/2 \right)} \approx {t_{mix}(\epsilon/2) \over \epsilon}$$
  Thus for $T=K t_{mix}(\epsilon/2)\approx {t_{mix}(\epsilon/2)^2 \over \epsilon}$,
 $$\max_{\nu} \| \nu P_{1 \over T} P_{2 \over T} \cdots P_{T-1 \over T} P_1 - \pi_f \|_{TV} \leq \epsilon$$
 \end{proof}
Observe that Theorem \ref{adiabatic1} is independent of the distance. We can use $\ell_2$ norm in the definitions of adiabatic and mixing times, and arrive to the same result.

 While a version of the quantum adiabatic theorem can be proved without a gap condition (see \cite{ae}), the Markov chain $P_{final}$ over finite sample space $\Omega$, if it is reversible, will have a spectral gap. Here we apply Theorem \ref{relax} to the result in Theorem \ref{adiabatic1}. 
 \begin{cor}  Suppose $P_{initial}$ and $P_{final}$ are  Markov chains with state space $\Omega$. If $P_{final}$ is reversible, irreducible and aperiodic with its spectral gap $\beta>0$, then
 $$T_{\epsilon}=O\left({\log{2 \over \epsilon}+\log{1 \over \min_{x \in \Omega}\pi_f(x)} \over \epsilon \beta^2}\right)$$
 \end{cor}

\section{Continuous time Markov processes}

 In the case of continuous time Markov chains, an equivalent result is produced via the method of uniformization and order statistics. Suppose $Q$ is a bounded Markov generator for a continuous time Markov chain $P(t)$, and \
 $\lambda \geq \max_{i \in \Omega} \sum_{j:j \not= i}q(i,j)$ is the upper bound on the departure rates over all states. The method of {\it uniformization} provides an expression for the transition probabilities $P(t)$ as follows:
 $$P(t)=\sum_{n=0}^{\infty} {(\lambda t)^n \over n!} e^{-\lambda t} P_{\lambda}^n,~~~~~\text{ where }~P_{\lambda}=I+{1 \over \lambda}Q$$
 The expression is obtained via conditioning on the number of arrivals in a Poisson process with rate $\lambda$.
 \vskip 0.2 in
 \noindent
 The definition of a Mixing time  is similar in the case of continuous time processes.
 \begin{Def}
 Suppose $P(t)$ is an irreducible and finite continuous time Markov chain with stationary distribution $\pi$. Given an $\epsilon >0$, the mixing time $t_{mix}(\epsilon)$ is defined as
 $$t_{mix}(\epsilon)=\inf\left\{t~:~\|\nu P(t) - \pi\|_{TV} \leq \epsilon,~~\text{ for all probability distributions } \nu \right\}$$
\end{Def}

 Suppose $Q_{initial}$ and $Q_{final}$ are two bounded generators for continuous time Markov processes over a finite state space $\Omega$, and $\pi_f$ is the only stationary distribution for $Q_{final}$.
   Let $$Q[s]=(1-s)Q_{initial}+sQ_{final}$$ be a time non-homogeneous  generator.
   Given $T>0$, let $P_T(t_1,t_2)$ ($0 \leq t_1 \leq t_2 \leq T$) denote a matrix of transition probabilities of a Markov process generated by $Q\left[{t \over T}\right]$ over $[t_1,t_2]$ time interval.
   
   \vskip 0.2 in
   \noindent
 Observe that the continuous time Markov adiabatic evolution is governed by 
 $${d \nu_t \over dt}=\nu_t Q\left[{t \over T}\right],~~~t \in [0,T]$$
 while the quantum adiabatic evolution is described via the corresponding Schr\"odinger's equation 
 ${d v_t \over dt}=-iv_t \mathcal{H}^T\left({t \over T}\right)$.

   \vskip 0.2 in  
   \begin{Def}
  Given $\epsilon>0$, a time $T_{\epsilon}$ is called the {\it adiabatic time} if it is the least $T$ such that
  $$\max_{\nu} \| \nu P_{T}(0,T) - \pi_f \|_{TV} \leq \epsilon,$$
  where the maximum is taken over all probability distributions $\nu$ over $\Omega$.  
  \end{Def} 
We will state and prove the following adiabatic theorem.
\vskip 0.2 in
\begin{thm} Let $t_{mix}$ denote the mixing time for $Q_{final}$. Take $\lambda$ such that
$$\lambda \geq \max_{i \in \Omega} \sum_{j:j \not= i}q_{initial}(i,j),~~\lambda \geq \max_{i \in \Omega} \sum_{j:j \not= i}q_{final}(i,j)~ 
\text{ and }~\lambda \geq {\epsilon \over 2t_{mix}(\epsilon/2)}+1,$$ 
where
$q_{initial}(i,j)$ and $q_{final}(i,j)$ are the rates in $Q_{initial}$ and $Q_{final}$ respectively.
Then the adiabatic time
$$T_{\epsilon} \leq {\lambda t_{mix}(\epsilon/2)^2 \over \epsilon} $$   
\end{thm}   
 \begin{proof}
 Observe that
 $\lambda \geq \max_{i \in \Omega} \sum_{j:j \not= i}q_t(i,j)$, where $q_t(i,j)$ are the rates in $Q\left[{t \over T}\right]$ ($0 \leq t \leq T$).
 
  Let $T=K t_{mix}(\epsilon/2)$ and $N=(K-1) t_{mix}(\epsilon/2)$, so that 
 $$\max_{\nu} \| \nu P_{final}(T-N)  - \pi_f \|_{TV} \leq \epsilon /2,$$
 where $P_{final}(t)=e^{tQ_{final}}$ denotes the transition probability matrix associated with the generator $Q_{final}$.
 
 Now, let $P_0=I+{1 \over \lambda}Q_{initial}$ and $P_1=I+{1 \over \lambda}Q_{final}$. Then $P_0$ and $P_1$ are discrete Markov chains and
 $$\nu P_{T}(0,T)  =  \nu_N P_T(N,T)= \nu_N \left( \sum_{n=0}^{\infty} {(\lambda (T-N)^n \over n!} e^{-\lambda (T-N)} I_n \right),$$
  where $\nu_N=\nu P_{T} (0,N)$ and
  $$I_n={n!  \over (T-N)^n}\int \!\!\!\dots\!\!\! \int _{N<x_1 <x_2 <\dots<x_n<T}[(1-x_1/T)P_0+(x_1/T)P_1]\dots[(1-x_n/T)P_0+(x_n/T)P_1] dx_1 \dots dx_n$$
  i.e. the order statistics of $n$ arrivals within the $[N,T]$ time interval. We used the fact that, when conditioned on the number of arrivals, the arrival times of a Poisson process are distributed as an order statistics of uniform random  variables.
  Hence
  \begin{eqnarray*}
 \nu P_{T}(0,T) & = & \nu_N \left( \sum_{n=0}^{\infty} (\lambda (T-N)^n \cdot {e^{-\lambda (T-N)}  \over (T-N)^n T^n n!} 
 P_1^n \int _N^T \dots \int_N^T x_1\dots x_n  dx_1 \dots dx_n \right)+ \mathcal{E} \\
 & = & e^{-\lambda t_{mix}(\epsilon/2)} \nu_N  \left( \sum_{n=0}^{\infty} {\lambda^n [(1-{1 \over 2K})  t_{mix}(\epsilon/2)]^n \over n!} P_1^n\right)+ \mathcal{E} \\
 & = & e^{-\lambda t_{mix}(\epsilon/2)} \nu_N  P_{final} \left(\lambda \left(1-{1 \over 2K}\right)  t_{mix}(\epsilon/2)\right)+ \mathcal{E} \\
 \end{eqnarray*}
  where  $\mathcal{E}$ is the rest of the terms. Thus, the total variation distance,
  $$\max_{\nu} \| \nu P_{T}(0,T) - \pi_f \|_{TV} \leq e^{-\lambda t_{mix}(\epsilon/2)}\epsilon/2+S_N$$
  whenever $\lambda \left(1-{1 \over 2K}\right) \geq 1$, i.e. $K \geq {\lambda \over 2 (\lambda-1)}$.  
  Taking $K \geq {\lambda t_{mix}(\epsilon/2) \over \epsilon}$, we bound the error term 
  $$S_N=\|\mathcal{E} - \pi_f \|_{TV} \leq1-e^{-\lambda t_{mix}(\epsilon/2)} \sum_{n=0}^{\infty} {\lambda^n [(1-{1 \over 2K})  t_{mix}(\epsilon/2)]^n \over n!}=1-e^{-{\lambda t_{mix}(\epsilon/2) \over 2K}} \leq \epsilon/2$$
  as $\epsilon < -2\log\left( 1 -{\epsilon \over 2}\right)$.
  
   \vskip 0.2 in
   \noindent
  The condition  $K \geq {\lambda \over 2 (\lambda-1)}$ is met since we have taken $\lambda \geq {\epsilon \over 2t_{mix}(\epsilon/2)}+1$ and\\ $K \geq {\lambda t_{mix}(\epsilon/2) \over \epsilon}$.  Therefore 
  $$T_{\epsilon} \leq {\lambda t_{mix}(\epsilon/2)^2 \over \epsilon} $$
  
 \end{proof}
 
 \noindent
 In the end, we would like to mention a possible application.
 The Glauber dynamics of a finite Ising model is a continuous time Markov process. Its mixing time can be estimated using path coupling. The adiabatic transformation of the Markov generator corresponds to an adiabatic transformation of the Hamiltonian. The above result can be applied to obtain the adiabatic time for the transformation. The result can be adjusted when the adiabatic evolution of the generator is nonlinear.

%\section*{Acknowledgment}

\bibliographystyle{amsplain}

\end{document}